\documentclass[a4paper,12pt]{article}
\usepackage[mathscr]{eucal}
\usepackage{amssymb}
\usepackage{latexsym}
\usepackage{amsthm}
\usepackage{amsmath}
\usepackage{a4wide}

\newtheorem{theorem}{Theorem}[section]

\newtheorem{lemma}[theorem]{Lemma}

\theoremstyle{definition}

\newtheorem{example}[theorem]{Example}

\makeatletter
\@addtoreset{equation}{section}

\makeatother

\title{Polynomial properties on large symmetric association schemes} 

\author{
Hiroshi Nozaki
}

\begin{document}
\maketitle

\renewcommand{\thefootnote}{\fnsymbol{footnote}}
\footnote[0]{2010 Mathematics Subject Classification: 
05E30 
(05B20).
}

\begin{abstract}
In this paper we characterize  ``large'' regular graphs using certain entries in
the projection matrices onto the eigenspaces of the graph. As a corollary  of this result, we show that ``large'' association schemes become $P$-polynomial association schemes. Our results are summarized as follows. 
Let $G=(V,E)$ be a connected $k$-regular graph with $d+1$ distinct eigenvalues $k=\theta_0>\theta_1>\cdots>\theta_d$. 
Since the diameter of $G$ is at most $d$, we have the Moore bound 
\[
|V| \leq M(k,d)=1+k \sum_{i=0}^{d-1}(k-1)^i. 
\]
Note that if $|V|> M(k,d-1)$ holds, the diameter of $G$ is equal to $d$. 
Let $E_i$ be the orthogonal projection matrix onto the eigenspace
corresponding to $\theta_i$. Let $\partial(u,v)$ be the path distance of $u,v \in V$. 

\noindent
{\bf Theorem.} {\it Assume $|V|> M(k,d-1)$ holds. Then for $x,y \in V$ with $\partial(x,y)=d$, 
 the $(x,y)$-entry of  $E_i$ is equal to 
\[
-\frac{1}{|V|}\prod_{j=1,2,\ldots,d, j \ne i} \frac{\theta_0-\theta_j}{\theta_i-\theta_j}. 
\]}
If a symmetric association scheme $\mathfrak{X}=(X,\{R_i\}_{i=0}^d)$ has a relation $R_i$ such that the graph $(X,R_i)$ satisfies the above condition, 
then $\mathfrak{X}$ is $P$-polynomial. Moreover we show the ``dual'' version of this theorem for spherical sets and $Q$-polynomial association  schemes. 
\end{abstract}

\textbf{Key words}: Polynomial association scheme, Moore bound, graph spectrum, 
 $s$-distance set, absolute bound. 

\section{Introduction}
A {\it symmetric association scheme} of class $d$ is a pair $\mathfrak{X} = (X, \{R_i \}_{i=0}^d)$, 
	where $X$ is a finite set and $\{R_i \}_{i=0}^d$ is a 
set of binary relations on $X$ satisfying 
	\begin{enumerate}
	\item $R_0 = \{(x, x) \mid x \in X \}$,
	\item $X \times X = \bigcup_{i=0}^d R_i$, and $R_i \cap R_j$ is empty if $i\ne j$, 
	\item $^tR_i = R_i$ for any $i \in\{ 0, 1,\ldots , d\}$, where $^t R_i = \{(y, x) \mid (x, y) \in R_i\}$, 
	\item for any $i,j,k\in \{0,1,\ldots ,d\}$, there exists an integer $p_{i j}^k$ such that for any pair $x,y \in X$ with $(x,y) \in R_k$, it holds that $p_{i j}^k= |\{ z \in X \mid (x,z) \in R_i, (z,y) \in R_j \}|$.  
	\end{enumerate}
The $i$-th {\it adjacency matrix} $A_i$ of $\mathfrak{X}$ is the matrix indexed by $X$ with 
the entry
\[
(A_i)_{xy}=
\begin{cases}
1 \text{ if $(x,y)\in R_i$}, \\
0 \text{ otherwise}. 
\end{cases}
\]
	The {\it Bose--Mesner algebra} $\mathfrak{A}$ of $\mathfrak{X}$ is the algebra generated 
	by the adjacency matrices $A_0, A_1, \ldots, A_d $ over the complex field $\mathbb{C}$. 
	Then $\{A_i\}_{i=0}^d$ is a natural basis of $\mathfrak{A}$, and $\mathfrak{A}$ is also closed under the 
Hadamard product, that is the entry-wise matrix product. 
	$\mathfrak{A}$ 
	has another remarkable basis which consists of  primitive idempotents $E_0, E_1, \ldots, E_d$ \cite[Section II.3]{BIb}.
	We define $P_i(j)$, $Q_i(j)$  
by the following equalities: 
	\[
	A_i= \sum_{j=0}^d P_i(j)E_j, \qquad
	E_i=\frac{1}{|X|}\sum_{j=0}^d Q_i(j) A_j.
\]
The values $P_i(j)$, $Q_i(j)$ are called the {\it parameters} of an association scheme. 
We use the notation  
$k_i=P_i(0)$ (degrees), and 
$m_i=Q_i(0)$  (multiplicities)
for $i=0,1,\ldots, d$.

A symmetric association scheme is called a {\it $P$-polynomial scheme} with respect to the ordering $A_0,A_1,\ldots, A_d $ 
	 if  
	there exists a polynomial $v_i$ of degree $i$ such that $A_i=v_i(A_1)$ for each $i \in \{0,1,\ldots, d \}$.
We say a symmetric association scheme is 
a $P$-polynomial scheme with respect to $A_i$
if it
has the $P$-polynomial property
with respect to some ordering
$A_0,A_i, A_{i_2},A_{i_3},\ldots ,A_{i_d}$.
	A symmetric association scheme is called a {\it $Q$-polynomial scheme}   
	with respect to the ordering $E_0,E_1, \ldots, E_d$ if  
	there exists a polynomial $v_i^{\ast}$ of degree $i$ 
	such that $E_i=v_i^{\ast}(E_1^{\circ})$ for each $i \in \{0,1,\ldots, d \}$, where $\circ$ means the multiplicity is the Hadamard product.  
	Moreover a symmetric association scheme is called
a $Q$-polynomial scheme with respect to $E_i$
if it
has the $Q$-polynomial property
with respect to some ordering
$E_0,E_i, E_{i_2},E_{i_3},\ldots ,E_{i_d}$.

The $P$-polynomial schemes and $Q$-polynomial 
schemes are interpreted as discrete cases of two-point homogeneous spaces and rank 1 symmetric spaces, respectively \cite[Section III.6]{BIb}, \cite[Chapter 9]{CSb}.  Wang \cite{W52} showed that compact two-point homogeneous spaces are   
compact rank 1 symmetric spaces and vise versa. 
Bannai and Ito \cite[Section III.6]{BIb} conjectured that if class $d$ is sufficiently large, 
a primitive association scheme is $P$-polynomial  
if and only if it is $Q$-polynomial. Here a symmetric association scheme $(X, \{R_i\}_{i=0}^d)$ is said to be {\it primitive} if the graph $(X,R_i)$ is 
connected for each $i=1,\ldots, d$. 
One of the main contributions is a sufficient condition for association schemes to have polynomial properties.  
Almost characterizations of the polynomial properties are proved by the relationship among the parameters on the scheme, see
\cite{BIb, BCNb, FG97, K12, KNX, KN12}. In this paper, 
we just focus on the size, and 
prove a sufficiently large association scheme has the polynomial property. 

There are two upper bounds for the size of a symmetric association scheme $\mathfrak{X}=(X,\{R_i\}_{i=0}^d)$.  
We have two interpretations of $\mathfrak{X}$ as regular graphs $A_i$ with eigenvalues $\{P_i(j)\}_{j=0}^d$, and spherical sets $E_i$ with inner products $\{Q_i(j)/|X|\}_{j=0}^d$.  If $A_i$ is connected, then $A_i$ is of diameter at most $d$, and  
we have the Moore bound. Namely if $P_i(0)$ is distinct from 
$P_i(1),P_i(2),\ldots,P_i(d)$, then we have 
\[
|X| \leq  M(k_i,d)=1+k_i\sum_{j=0}^{d-1}(k_i-1)^j.
\]
On the other hand, if the diagonal entries of $E_i$ are distinct from the others, $E_i$ becomes the Gram matrix of some spherical finite set with 
at most $d$ distances between distinct points. 
We have an upper bound for the cardinality of a spherical finite set with 
only $s$ distances \cite{DGS77}. Namely if $Q_i(0)$ is distinct from 
$Q_i(1),Q_i(2),\ldots,Q_i(d)$,
then
\[|X| \leq N(m_i,d) = \binom{m_i+d-1}{d}+\binom{m_i+d-2}{d-1}.\]
In the present paper, we show the 
following:  
\begin{enumerate}
\item[(i)]  If $P_i(0)$ is distinct from 
$P_i(1),P_i(2),\ldots,P_i(d)$ and $|X| >  M(k_i,d-1)$, then $(X,\{R_i\}_{i=0}^d)$ has the 
$P$-polynomial property with respect to $A_i$. 
\item[(ii)] If $Q_i(0)$ is distinct from 
$Q_i(1),Q_i(2),\ldots,Q_i(d)$ and 
$|X| > N(m_i,d-1)$, 
then $(X,\{R_i\}_{i=0}^d)$ has the 
$Q$-polynomial property with respect to $E_i$.
\end{enumerate}
Though these sufficient conditions are fairly simple, there are many 
examples including strongly regular graphs, Johnson or Hamming schemes of 
sufficiently large degree or multiplicity.

\section{Regular graphs and $P$-polynomial schemes} \label{sec_2}
Let $G=(V,E)$ be a connected $k$-regular graph with at most $d+1$ 
distinct eigenvalues, and $A$ the adjacency matrix. 
Since the diameter of $G$ is at most $d$,   
we have the Moore bound 
  \[
|V| \leq M(k,d)=1+k\sum_{j=0}^{d-1}(k-1)^j. 
\] 
Let $k=\theta_0>\theta_1>\theta_2>\cdots>\theta_s$ be distinct eigenvalues of $G$, 
where $s\leq d$. 
Let $E_i$ be the orthogonal projection
matrix onto the eigenspace corresponding to $\theta_i$. In particular $E_0=(1/|V|)J$, where $J$ is 
the all-ones matrix. 
Let $\partial(x,y)$ be the path distance of $x \in V$ and $y \in V$.  
Let $R_i=\{(x,y) \mid x,y \in V,\partial(x,y)=i \}$ and $R_i(x)=\{ y \mid y \in V, \partial(x,y)=i \}$.  
For each $i \in \{1,\ldots, d\}$, we define
\[
K_i= \prod_{j=1,2,\ldots,s, j\ne i} \frac{\theta_0-\theta_j}{\theta_i-\theta_j}.
\] 

We begin with the following result.
\begin{theorem} \label{thm:con_gra}
Let $G=(V,E)$ be a connected 
$k$-regular graph with diameter $d$. 
If $G$ has precisely $d+1$ distinct eigenvalues 
then, for $(x,y) \in R_d$, the $(x,y)$-entry of $E_i$ is $-K_i/|V|$ for each $i\in \{ 1,2,\ldots,d\}$. 
\end{theorem}
\begin{proof}
Define  
\[
f_i(t)=\prod_{j=1,2,\ldots,d, j\ne i} \frac{t-\theta_j}{\theta_i-\theta_j}
\] 
for each $i \in \{1,2,\ldots, d\}$. Then we have
\[
f_i(A)=\sum_{j=0}^d f_i(\theta_j) E_j= K_i E_0 + E_i. 
\]
Because the degree of $f_i(t)$ is $d-1$ and $G$ is of diameter $d$,  the $(x,y)$-entry of $f_i(A)$ is equal to $0$ for $(x,y) \in R_d$. Therefore the $(x,y)$-entry of $E_i$ is equal to $-K_i/|V|$ for each $i \in \{1,2,\ldots, d\}$. 
\end{proof}
We now use the Moore bound to show that a large connected regular graph satisfies the assumption of Theorem~\ref{thm:con_gra}.
\begin{theorem}\label{thm:gen_P}
Let $G=(V,E)$ be a connected 
$k$-regular graph 
 with at most $d+1$ distinct eigenvalues. 
Assume $|V|> M(k,d-1)$ holds. Then we have the following. 
\begin{enumerate} 
\item $G$ is of diameter $d$.
\item $G$ has $d+1$ distinct eigenvalues. 
\item For $(x,y) \in R_d$, the $(x,y)$-entry of $E_i$ is $-K_i/|V|$ for each $i\in \{1,\ldots,d\}$. 
\item For each $x\in V$, the number of entries $-K_i/|V|$ in the 
$x$-th row of $E_i$ is at least $|V|-M(k,d-1)$.  
\end{enumerate}
\end{theorem}
\begin{proof}
(1): The diameter of $G$ is at most $d$ because $G$ has at most $d+1$ distinct eigenvalues. 
If the diameter of $G$ is smaller than $d$, then we have $|V|\leq M(k,d-1)$, a contradiction. Therefore the diameter of $G$ is $d$. 
 
(2): 
From (1), $G$ has at least $d+1$ distinct eigenvalues. Therefore $G$ has exactly $d+1$ distinct eigenvalues. 

(3): This follows immediately from (1), (2), and Theorem~\ref{thm:con_gra}. 

(4):
For each $x \in V$, we have $|\bigcup_{i=0}^{d-1} R_i(x)| \leq M(k,d-1)$. Therefore $|R_d(x)|\geq |V|- M(k,d-1)$ holds. This implies that 
the number of entries $-K_i/|V|$ in the 
$x$-th row of $E_i$ is at least $|V|-M(k,d-1)$. 
\end{proof}

We apply the above theorem to symmetric association schemes. First we recall the following easy fact. 
\begin{lemma}[{\cite[Lemma 3.2, page 229]{Gb}} or  {\cite[Lemma 3.2]{KNX}}] \label{thm:known}
Let $\mathfrak{X}=(X,\{R_i\}_{i=0}^d)$ be a symmetric association scheme of class $d$. 
Suppose that $P_j(0)$ is distinct from 
$P_j(i)$ for $i=1,\ldots,d$. 
Then $(X,R_j)$ has diameter $d$ if and only if $\mathfrak{X}$ is $P$-polynomial with respect to $A_j$
\end{lemma}
By Theorem \ref{thm:gen_P} and Lemma \ref{thm:known}, 
we immediately obtain the following theorem which shows the $P$-polynomial property of a large association scheme. 
\begin{theorem}
\label{thm:P-poly2}
Let $\mathfrak{X}=(X,\{R_i\}_{i=0}^d)$ be a symmetric association scheme. Suppose  $P_j(0)$ is distinct from $P_j(1),\ldots,P_j(d)$.  
 If $|X| > M(k_j,d-1)$ holds,   
then $\mathfrak{X}$ is 
a $P$-polynomial association scheme with respect to $A_j$.
\end{theorem}

We also remark that Theorem~\ref{thm:con_gra}, together with Lemma~\ref{thm:known}, generalizes the implication $(1)\Rightarrow (2)$ of the following known characterization of $P$-polynomial association schemes. 
\begin{theorem}[\cite{KNX,NT11}]
\label{thm:P-poly}
Let $\mathfrak{X}=(X,\{R_i\}^d_{i=0})$ be a symmetric
association scheme of class $d$.
Suppose that $P_j(0),P_j(1),\ldots,P_j(d)$
are mutually distinct.
Then the following are equivalent:
\begin{enumerate}
\item $\mathfrak{X}$ is a $P$-polynomial association scheme
with respect to  $A_j$.
\item There exists $l\in \{0,1,\ldots ,d\}$ such that
for each $h \in \{1,2,\ldots ,d\}$,
\begin{equation*}
\prod_{i=1,2,\ldots,d, 
i \neq h}
\frac{P_j(0)-P_j(i)}{P_j(h)-P_j(i)}= -Q_h(l).
\end{equation*}
\end{enumerate}
Moreover if $($2$)$ holds, then $A_l$ is the $d$-th matrix with respect to the 
resulting polynomial ordering.
\end{theorem}

Though the condition in Theorem~\ref{thm:P-poly2} is fairly simple, there are many examples. 

\begin{example}
Every connected strongly regular graph with $v$ vertices satisfies the assumption
 in Theorem~\ref{thm:P-poly2} because of $v>M(k,1)=1+k$. 
\end{example}
\begin{example}
Let $G$ be a connected regular graph of girth $g$, with $d+1$ distinct eigenvalues, and $v$ vertices,  which satisfies $g\geq 2d-1$. It is known that $G$ is distance-regular \cite{NX, ADFX}.  
We have the lower bound $v \geq M(k,d-1)$ from the assumption of girth  \cite{T66}. 
If $G$ attains this bound, then $G$ is a Moore graph with only $d$ distinct eigenvalues. Therefore $v > M(k,d-1)$ holds, and  
$G$ satisfies the condition in Theorem~\ref{thm:P-poly2}. 
Known examples are listed in \cite{NX}. 
For instance, a Moore graph satisfies $g\geq 2d-1$. 
\end{example}  
\begin{example}
Infinite families of distance-regular graphs of diameter $d$ with unbounded degree $k$ satisfy the condition
 in Theorem~\ref{thm:P-poly2} for sufficiently large $k$.  Indeed the number of the vertices can be expressed by the polynomial 
$\sum_{i=0}^d v_i(k)$ in $k$ of degree $d$.  
For example, the Johnson scheme $J(n,3)$ satisfies the condition for every $n \geq 51$, and 
the Hamming scheme $H(3,q)$ satisfies the condition for every $q\geq 7$.
\end{example}
\label{sec:2}

\section{Spherical sets and $Q$-polynomial schemes}
 The dual results of those in 
Section \ref{sec:2} will be obtained in this section. 
Let $A(X)=\{\langle x,y \rangle \mid x, y \in X, x\ne y\}$ for $X$ in the
unit sphere $S^{m-1}$, where $\langle x,y \rangle$
be the usual inner product of $x$ and $y$. 
Let $X$ be a finite subset in $S^{m-1}$ which satisfies $|A(X)|\leq d$, where $d$ is not as in the previous section. 
 Let $\theta^*_1>\cdots > \theta^*_s$ be the elements in $A(X)$, and $\theta^*_0=1$, where $s\leq d$ and $\theta_0^*>\theta^*_1>\cdots > \theta^*_s$. Then we have 
the absolute bound \cite{DGS77}: 
\begin{equation*}
|X| \leq N(m,d)=\binom{m+d-1}{d}+\binom{m+d-2}{d-1}.
\end{equation*}
We can obtain the graph $G_i=(X,R_i)$,  
where $R_i=\{(x,y)  \mid x,y\in X, \langle x,y \rangle=\theta^*_i  \}$ for each $i \in \{1,2, \ldots,d\}$.  
Let $A_i$ be the adjacency matrix of $G_i$. 
For each $i \in \{1,2,\ldots,d\}$, we define 
\[
K_i^*= \prod_{j=1,2,\ldots,s, j\ne i} \frac{\theta^*_0-\theta^*_j}{\theta^*_i-\theta^*_j}. 
\]
We say an $n \times n$ matrix $E$ is {\it Schur-connected} 
if there is a polynomial $q$ such that $q(E^{\circ})$ has 
rank $n$ \cite{Gb}, where $\circ$ means the Hadamard product. Note that the Gram matrix $M$ of a finite set $X$ in $S^{m-1}$ is Schur-connected by taking $q(x)$ as the annihilator $\prod_{\alpha \in A(X)}(x-\alpha)$. 
The {\it Schur-diameter} of  $E$ is the least integer $d$ 
such that $q(E^{\circ})$ has rank $n$ for some polynomial $q$ of degree $d$ \cite{Gb}. If the Schur-diameter of $M$  is $d$, then 
$|A(X)|  \geq d$.  Let $P_i(S^{m-1})$ denote the linear space of the restrictions of polynomials of 
degree at most $i$, in $m$ variables, to $S^{m-1}$.

We begin with the following result.  
\begin{theorem} \label{thm:sph_set}
Let $X$ be a finite set in $S^{m-1}$ which satisfies the Schur-diameter of the Gram matrix $M$ is equal to $d$. If $|A(X)|=d$ then $-K_i^*$ is an eigenvalue of $A_i$ for each $i \in \{1,2,\ldots, d\}$. Moreover the multiplicity 
of $-K_i^*$ is at least $|X|-N(m,d-1)$.
\end{theorem}
\begin{proof}
Define  
\[
f_i^*(t)=\prod_{j=1,2,\ldots,d, j\ne i} 
\frac{t-\theta^*_j}{\theta^*_i-\theta^*_j}
\] 
for each $i \in \{1,2,\ldots, d\}$. Note
that every diagonal entry of $M$ is $1$.  
Then we have
\begin{align} \label{eq:fi}
f_i^*(M^\circ)= K_i^* I + A_i.
\end{align} 
For each $x\in X$ and $m$ variables $\xi=(\xi_1,\ldots, \xi_m)$, we consider 
a polynomial $f_i^*(\langle x,\xi \rangle )$. 
By Lemma~2.2 in \cite{N11}, the rank of $K_i^*I+A_i=(f_i^*(\langle x,y \rangle ))_{x,y \in X}$ is bounded above by $N(m,d-1)=\dim P_{d-1}(S^{m-1})$. 
Therefore the matrix \eqref{eq:fi} has  
eigenvalue zero with multiplicity at least $|X|-N(m,d-1)$. Thus $A_i$ has the eigenvalue $-K_i^*$
with multiplicity at least $|X|-N(m,d-1)$.
\end{proof}

We use the absolute bound to show 
that a large spherical set satisfies the assumption of Theorem~\ref{thm:sph_set}. 
\begin{theorem} \label{thm:main_sph}
Let $X$ be a finite set in $S^{m-1}$ which satisfies $|A(X)|\leq d$.
Assume $|X|>N(m,d-1)$ holds. Then we have the following.
\begin{enumerate}
\item The Schur-diameter of the Gram matrix $M$ of $X$ is equal to $d$. 
\item $X$ has $d$ inner  products, namely $|A(X)|=d$.  
\item $-K_i^*$ is an eigenvalue of $A_i$ for each $i \in \{1,2,\ldots, d\}$.
\item The multiplicity of $-K_i^*$ is at least 
$|X|-N(m,d-1)$.
\end{enumerate}  
\end{theorem}
\begin{proof}
(1): Since  $|A(X)|\leq d$ holds, the Schur-diameter of $M$ 
is at most $d$. 
If the Schur-diameter of $M$ is less than $d$, then
there exists a polynomial $q(x)$ of degree less than $d$ such 
that the rank of $q(M^{\circ})$ is $|X|$. 
Since $q(\langle x,\xi \rangle ) \in P_{d-1}(S^{m-1})$ for $x\in S^{m-1}$,  in variables  
 $\xi=(\xi_1, \ldots, \xi_m)$,  we have $|X| \leq N(m,d-1)=\dim P_{d-1}(S^{m-1})$ by Lemma~2.2 in \cite{N11}, a contradiction. Therefore the Schur-diameter of the Gram matrix $M$ of $X$ is equal to $d$.

(2): From (1), we have $|A(X)| \geq d$, hence $|A(X)|=d$.  

(3),(4): These follow immediately 
from (1), (2), and Theorem~\ref{thm:sph_set}. 
\end{proof}

We apply the above theorems to symmetric association schemes. 
First we recall the following fact. 
\begin{lemma}[{\cite[Lemma 2.2]{KN12}}] \label{thm:known2}
Let $\mathfrak{X}=(X,\{R_i\}_{i=0}^d)$ be a symmetric association scheme of class $d$. 
Suppose that $Q_j(0)$ is distinct from 
$Q_j(i)$ for $i=1,\ldots,d$. 
Then $E_j$ has Schur-diameter $d$ if and only if $\mathfrak{X}$ is $Q$-polynomial with respect to $E_j$. 
\end{lemma}

By Theorem~\ref{thm:main_sph} and Lemma~\ref{thm:known2}, 
we immediately obtain the following theorem which shows the $Q$-polynomial property of a large symmetric association scheme. 
\begin{theorem}\label{thm:Q-poly}
Let $\mathfrak{X}=(X,\{R_i\}_{i=0}^d)$ be a symmetric association scheme. Assume $Q_j(0)$ is 
distinct from $Q_j(1),\ldots,Q_j(d)$. 
If $|X| > N(m_j,d-1)$, then $\mathfrak{X}$ is 
a $Q$-polynomial  scheme with respect to $E_j$.
\end{theorem}

We also remark that Theorem~\ref{thm:sph_set}, together with Lemma~\ref{thm:known2}, 
generalizes the implication $(1) \Rightarrow (2)$ of the following known 
characterization of $Q$-polynomial association schemes. 

\begin{theorem}[\cite{KN12,NT11}]
\label{thm:Q-poly2}
Let $\mathfrak{X}=(X,\{R_i\}^d_{i=0})$ be a symmetric
association scheme of class $d$.
Suppose that $Q_j(0),Q_j(1),\ldots,Q_j(d)$
are mutually distinct.
Then the following are equivalent:
\begin{enumerate}
\item $\mathfrak{X}$ is a $Q$-polynomial association scheme
with respect to  $E_j$.
\item There exists $l\in \{0,1,\ldots ,d\}$ such that
for each $h \in \{1,2,\ldots ,d\}$,
\begin{equation*}
\prod_{i=1,2, \ldots, d,
i \neq h}
\frac{Q_j(0)-Q_j(i)}{Q_j(h)-Q_j(i)}= -P_h(l).
\end{equation*}
\end{enumerate}
Moreover if $($2$)$ holds, then $E_l$ is the $d$-th matrix with respect to the 
resulting polynomial ordering.
\end{theorem}

The following are examples satisfying the condition in Theorem~\ref{thm:Q-poly}.

\begin{example}
Every connected strongly regular graph with $v$ vertices satisfies the assumption
 in Theorem~\ref{thm:Q-poly} because of $v> N(m_1,1)=1+m_1$. 
\end{example}

\begin{example}
Let $X\subset S^{m-1}$ be a spherical $t$-design \cite{DGS77} with $d$ distances which satisfies $t\geq 2d-2$.  
Then the set $X$ has the structure of a $Q$-polynomial association scheme \cite{DGS77}. 
From the inequality $t\geq 2d-2$, we have $|X| \geq N(m, d-1)$ \cite{DGS77}, that is called an absolute bound for spherical designs. If $|X| = N(m, d-1)$ holds, then $X$  has only $d-1$ distances \cite{DGS77}. Thus $|X| > N(m, d-1)$ holds. 
The association scheme obtained from $X$ satisfies the assumption in Theorem~\ref{thm:Q-poly}. 
Known examples are listed in \cite{CK07}. 
For instance, a tight spherical design satisfies $t\geq 2d-2$.    
\end{example}

\begin{example}
Infinite families of $Q$-polynomial association schemes  with unbounded multiplicity $m_1$ satisfy the condition
 in Theorem~\ref{thm:Q-poly} for sufficiently large $m_1$.  Indeed the number of the vertices can be expressed by the polynomial 
$\sum_{i=0}^d v_i^*(|X|m_1)$ in $m_1$ of degree $d$. 
For example, the Johnson scheme $J(n,3)$ satisfies the condition for every $n \geq 7$, and 
the Hamming scheme $H(3,q)$ satisfies the condition for every $q\geq 4$.
\end{example}

\textbf{Acknowledgments.} 
The author thanks to Eiichi Bannai, Paul Terwilliger, and
Sho Suda for fruitful discussions. The author also wishes to thank two anonymous 
referees for their valuable suggestions that helped him to substantially improve this paper.  This work was supported by JSPS KAKENHI Grant Numbers 25800011, 26400003.

\noindent
{\it Hiroshi Nozaki}\\
	Department of Mathematics Education, \\
	Aichi University of Education\\
	1 Hirosawa, Igaya-cho, \\
	Kariya, Aichi 448-8542, \\
	Japan.\\ 
	hnozaki@auecc.aichi-edu.ac.jp\\

\end{document}